\documentclass{article}
\usepackage{amsmath,amsthm,amssymb,amsfonts}
\usepackage{verbatim}
\usepackage{graphicx}
\usepackage{stmaryrd} 
\usepackage{hyperref}

\usepackage[colorinlistoftodos]{todonotes}

\usepackage{tikz}
\usepackage{tkz-berge, tkz-graph}
\tikzset{vertex/.style={circle,draw,fill,inner sep=0pt,minimum size=1mm}}

\newcommand{\yzgrid}[5] 
{
  \foreach \y in {#2,...,#3} {
      \draw[white,line width=3pt] (\y,#4,#1) -- (\y,#5,#1);
  }
  \foreach \z in {#4,...,#5} {
      \draw[white,line width=3pt] (#2,\z,#1) -- (#3,\z,#1);
      }
  \foreach \y in {#2,...,#3} {
      \draw[densely dotted] (\y,#4,#1) -- (\y,#5,#1);
  }
  \foreach \z in {#4,...,#5} {
      \draw[densely dotted] (#2,\z,#1) -- (#3,\z,#1);
      }
}

\newcommand{\xyzgrid}[6]
{ 
\foreach \x in {#1,...,#2} {
  \foreach \y in {#3,...,#4} {
      \draw [densely dotted] (\y, #5, \x) -- (\y, #6, \x);
  }
}
\foreach \y in {#3,...,#4} {
    \foreach \z in {#5,...,#6} {
      \draw [densely dotted] (\y, \z, #1) -- (\y, \z, #2);
    }
}
\foreach \x in {#1,...,#2} {
    \foreach \z in {#5,...,#6} {
      \draw [densely dotted] (#3, \z, \x) -- (#4, \z, \x);
    }
}
}

\theoremstyle{plain}
\newtheorem{thm}{Theorem}

\newtheorem{cor}[thm]{Corollary}

\theoremstyle{definition}
\newtheorem{definition}[thm]{Definition}

\numberwithin{thm}{section}

\newcommand{\adj}{\leftrightarrow}
\newcommand{\adjeq}{\leftrightarroweq}

\def\Z{{\mathbb Z}}

\begin{document}
\title{$MSS_{18}$ is Digitally 18-contractible}
\author{Laurence Boxer
\thanks{
    Department of Computer and Information Sciences,
    Niagara University,
    Niagara University, NY 14109, USA;
    and Department of Computer Science and Engineering,
    State University of New York at Buffalo.
    email: boxer@niagara.edu
}
}

\date{ }
\maketitle{}

\begin{abstract}
The paper~\cite{Han06} incorrectly asserts that the digital image $MSS_{18}$,
a digital model of the Euclidean 2-sphere $S^2$, is not 18-contractible. We show this
assertion is false.

\end{abstract}


\section{Introduction}
In digital topology, we often find that properties of a digital image are
analogous to topological properties of an object in Euclidean space modeled by the 
digital image. For example, a digital image that models a contractible object may
have the property of digital contractibility, and a digital image that models a
non-contractible object may have the property of digital non-contractibility.

$MSS_{18}$ is the name often used for a certain digital image that models the
Euclidean 2-sphere $S^2$. S.E. Han has claimed (Theorem~4.3 of~\cite{Han06})
that $MSS_{18}$ is not 18-contractible. We show this assertion is false.

\section{Preliminaries}
Much of this section is quoted or paraphrased from~\cite{BxSt16}.

We use $\Z$ to indicate the set of integers.

\subsection{Adjacencies}
A digital image is a graph $(X,\kappa)$, where $X$ is a subset of $\Z^n$ for
some positive integer~$n$, and $\kappa$ is an adjacency relation for the points
of~$X$. The $c_u$-adjacencies are commonly used.
Let $x,y \in \Z^n$, $x \neq y$, where we consider these points as $n$-tuples of integers:
\[ x=(x_1,\ldots, x_n),~~~y=(y_1,\ldots,y_n).
\]
Let $u \in \Z$,
$1 \leq u \leq n$. We say $x$ and $y$ are 
{\em $c_u$-adjacent} if
\begin{itemize}
\item There are at most $u$ indices $i$ for which 
      $|x_i - y_i| = 1$.
\item For all indices $j$ such that $|x_j - y_j| \neq 1$ we
      have $x_j=y_j$.
\end{itemize}
Often, a $c_u$-adjacency is denoted by the number of points
adjacent to a given point in $\Z^n$ using this adjacency.
E.g.,
\begin{itemize}
\item In $\Z^1$, $c_1$-adjacency is 2-adjacency.
\item In $\Z^2$, $c_1$-adjacency is 4-adjacency and
      $c_2$-adjacency is 8-adjacency.
\item In $\Z^3$, $c_1$-adjacency is 6-adjacency,
      $c_2$-adjacency is 18-adjacency, and $c_3$-adjacency
      is 26-adjacency.
\end{itemize}

We write $x \adj_{\kappa} x'$, or $x \adj x'$ when $\kappa$ is understood, to indicate
that $x$ and $x'$ are $\kappa$-adjacent. Similarly, we
write $x \adjeq_{\kappa} x'$, or $x \adjeq x'$ when $\kappa$ is understood, to indicate
that $x$ and $x'$ are $\kappa$-adjacent or equal.

A subset $Y$ of a digital image $(X,\kappa)$ is
{\em $\kappa$-connected}~\cite{Rosenfeld},
or {\em connected} when $\kappa$
is understood, if for every pair of points $a,b \in Y$ there
exists a sequence $\{y_i\}_{i=0}^m \subset Y$ such that
$a=y_0$, $b=y_m$, and $y_i \adj_{\kappa} y_{i+1}$ for $0 \leq i < m$.

\subsection{Digitally continuous functions}
The following generalizes a definition of~\cite{Rosenfeld}.

\begin{definition}
\label{continuous}
{\rm ~\cite{Boxer99}}
Let $(X,\kappa)$ and $(Y,\lambda)$ be digital images. A single-valued function
$f: X \rightarrow Y$ is $(\kappa,\lambda)$-continuous if for
every $\kappa$-connected $A \subset X$ we have that
$f(A)$ is a $\lambda$-connected subset of $Y$. $\Box$
\end{definition}

When the adjacency relations are understood, we will simply say that $f$ is \emph{continuous}. Continuity can be expressed in terms of adjacency of points:
\begin{thm}
{\rm ~\cite{Rosenfeld,Boxer99}}
A function $f:X\to Y$ is continuous if and only if $x \adj x'$ in $X$ implies 
$f(x) \adjeq f(x')$. \qed
\end{thm}

See also~\cite{Chen94,Chen04}, where similar notions are referred to as {\em immersions}, {\em gradually varied operators},
and {\em gradually varied mappings}.

A homotopy between continuous functions may be thought of as
a continuous deformation of one of the functions into the 
other over a finite time period.

\begin{definition}{\rm (\cite{Boxer99}; see also \cite{Khalimsky})}
\label{htpy-2nd-def}
Let $X$ and $Y$ be digital images.
Let $f,g: X \rightarrow Y$ be $(\kappa,\kappa')$-continuous functions.
Suppose there is a positive integer $m$ and a function
$F: X \times [0,m]_{{\Z}} \rightarrow Y$
such that

\begin{itemize}
\item for all $x \in X$, $F(x,0) = f(x)$ and $F(x,m) = g(x)$;
\item for all $x \in X$, the induced function
      $F_x: [0,m]_{{\Z}} \rightarrow Y$ defined by
          \[ F_x(t) ~=~ F(x,t) \mbox{ for all } t \in [0,m]_{{\Z}} \]
          is $(2,\kappa')-$continuous. That is, $F_x(t)$ is a path in $Y$.
\item for all $t \in [0,m]_{{\Z}}$, the induced function
         $F_t: X \rightarrow Y$ defined by
          \[ F_t(x) ~=~ F(x,t) \mbox{ for all } x \in  X \]
          is $(\kappa,\kappa')-$continuous.
\end{itemize}
Then $F$ is a {\rm digital $(\kappa,\kappa')-$homotopy between} $f$ and
$g$, and $f$ and $g$ are {\rm digitally $(\kappa,\kappa')-$homotopic in} $Y$.
$\Box$
\end{definition}

If there is a $(\kappa,\kappa)$-homotopy $F: X \times [0,m]_{\Z} \to X$ between
the identity function $1_X$ and a constant function, we say 
$F$ is a (digital) {\em  $\kappa$-contraction} and $X$ is {\em $\kappa$-contractible}.

\section{Contractibility of $MSS_{18}$}
$MSS_{18}$~\cite{Han06} is a ``small" digital model of the Euclidean 2-sphere~$S^2$,
appearing rather like an American football.
As shown in Figure~\ref{MSS18fig}, we can take $MSS_{18} = \{p_i\}_{i=0}^9$, where
\[ p_0= (0,0,0),~p_1=(1,1,0),~p_2=(1,2,0),~p_3=(0,3,0),~p_4=(-1,2,0),
\]
\[ ~p_5=(-1,1,0),~p_6=(0,1,-1),~p_7=(0,2,-1),~p_8=(0,2,1),~p_9=(0,1,1).
\]

\begin{figure}
\begin{center}
\begin{tabular}{c}
\begin{tikzpicture}
%
\def \pts{0/0/0/p_0,1/1/0/p_1,1/2/0/p_2,0/3/0/p_3,-1/2/0/p_4,-1/1/0/p_5,0/1/-1/p_6,0/2/-1/p_7,0/2/1/p_8,0/1/1/p_9};
\xyzgrid{-1}{1}{0}{3}{-1}{1}
\foreach \x/\y/\z/\l in \pts {
            \node at (\y,\z,\x) [vertex, fill=black, label=$\l$,] {};
}
\end{tikzpicture}%
\end{tabular}
\end{center}
\caption{$MSS_{18}$, a digital model of a 2-sphere (from Figure~2 of~\cite{BxSt16})}
\label{MSS18fig}
\end{figure}
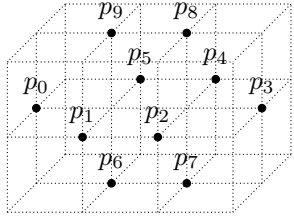

Contrary to the claim of Theorem~4.3 of~\cite{Han06}, we have the following 
Theorem~\ref{MSS18contracts}. Its proof makes use of the contractibility of a
4-point digital simple closed curve~\cite{Boxer99}. Notice that $MSS_{18}$ contains
the 4-point 18- and 26-simple closed curves
\[ S=\{(x,1,z) \in MSS_{18}\}=\{p_1,p_6,p_5,p_9\} \mbox{ and }\]
\[ S'=\{(x,2,z) \in MSS_{18}\} = \{p_2,p_7,p_4,p_8\}.
\]
Roughly, our contraction of $MSS_{18}$ begins by continuously deforming $MSS_{18}$
into a connected subset of $S \cup S'$, after which a contraction is completed.

\begin{thm}
\label{MSS18contracts}
$MSS_{18}$ is 18-contractible and 26-contractible.
\end{thm}

\begin{proof}
We define a contraction $H: MSS_{18} \times [0,3] \to MSS_{18}$ as follows.
\begin{itemize}
    \item For the step at time $t=0$, we let $H(p_i,0) = p_i$ for all members of
          $\{i\}_{i=0}^9$.
    \item For the step $t=1$, we let
          \[ H(p_i,1) = \left \{ \begin{array}{ll}
             p_1 & \mbox{if } i \in \{0,1,9\}; \\
             p_6 & \mbox{if } i \in \{5,6\}; \\
             p_2 & \mbox{if } i \in \{2,3,8\}; \\
             p_7  & \mbox{if } i \in \{4,7\}. \\
          \end{array} \right .
          \]
          Thus, during this step, $H$ begins contracting $S$, deforming $S$ to
          $\{p_1,p_6\}$; and also begins
          contracting $S'$, deforming $S$ to $\{p_2,p_7\}$; as well as bringing $p_0$ to $p_1$ and $p_3$ to $p_2$.
    \item For the step $t=2$, let
           \[ H(p_i,2) = \left \{ \begin{array}{ll}
              p_6 & \mbox{if }  H(p_i,1) \in \{p_1,p_6\}; \\
              p_7 & \mbox{if } H(p_i,1) \in \{p_2,p_7\}.
            \end{array} \right .
          \]
          This step completes the contraction of $S$ to the point $p_6$; it also
          completes the contraction of $S'$ to the point $p_7$.
    \item For the step $t=3$, let $H(p_i)=p_6$ for all indices~$i$.
\end{itemize}
It is elementary to verify that $H$ is an 18-homotopy and a 26-homotopy
between the identity on $MSS_{18}$ and a constant map.
\end{proof}

Theorem~\ref{MSS18contracts} adds to our 
knowledge~\cite{Boxer99,Boxer06} of ``small" digital spheres that are 
digitally contractible. It seems likely that ``large" digital spheres are not
digitally contractible, although other than for digital 1-spheres, i.e., 
simple closed curves~\cite{Boxer10}, at the current writing
the literature lacks results to support this conjecture.

Note also that since a contractible digital image has trivial 
fundamental group (\cite{Boxer05} - proof corrected in
\cite{BxSt18}), the following assertion, originally appearing
as Propositions~3.3 and~3.5 of~\cite{BxSt16}, is an
immediate consequence of Theorem~\ref{MSS18contracts}.

\begin{cor}
Let $x \in MSS_{18}$. Then the fundamental groups $\Pi_1^{18}(MSS_{18},x)$
and $\Pi_1^{26}(MSS_{18},x)$ of $(MSS_{18},x)$ with respect to 18- and 26-adjacency,
respectively, are trivial. $\qed$
\end{cor}

\section{Further remarks}
We have corrected an error of~\cite{Han06} by showing that $MSS_{18}$ is
contractible with respect to 18-adjacency.


\end{document}